\documentclass[12pt, reqno]{amsart}
\usepackage{amsmath, amsthm, amscd, amsfonts, amssymb, graphicx, color, mathrsfs}
\usepackage[bookmarksnumbered, colorlinks, plainpages]{hyperref}
\hypersetup{colorlinks=true,linkcolor=red, anchorcolor=green, citecolor=cyan, urlcolor=red, filecolor=magenta, pdftoolbar=true}
\usepackage{hyperref}
\newtheorem{theorem}{Theorem}[section]
\newtheorem{lemma}[theorem]{Lemma}

\theoremstyle{definition}
\newtheorem{definition}[theorem]{Definition}
\newtheorem{example}[theorem]{Example}

\newtheorem{problem}[theorem]{Problem}

\theoremstyle{remark}

\numberwithin{equation}{section}

\begin{document}

\setcounter{page}{1}

\title[Inverse coefficient problems for the heat equation ... ]{Inverse coefficient problems for the heat equation with fractional Laplacian
}

\author[A.  Mamanazarov and  D. Suragan]{Azizbek Mamanazarov and Durvudkhan Suragan}

\address{\textcolor[rgb]{0.00,0.00,0.84}{Durvudkhan Suragan:
Department of Mathematics, 
Nazarbayev University,
Astana 010000,
Kazakhstan
}}
\email{\textcolor[rgb]{0.00,0.00,0.84}{durvudkhan.suragan@nu.edu.kz}}

\address{\textcolor[rgb]{0.00,0.00,0.84}{Azizbek Mamanazarov:
Department of Mathematical Analysis and Differential Equations, Fergana State University, Fergana 150100, Uzbekistan}}
\email{\textcolor[rgb]{0.00,0.00,0.84}{mamanazarovaz1992@gmail.com; a.mamanazarov@pf.fdu.uz}}

\thanks{All authors contributed equally to the manuscript and read and approved the final manuscript.}

\let\thefootnote\relax\footnote{$^{*}$Durvudkhan Suragan}



\begin{abstract} In the present paper we study inverse problems related to determining the time-dependent coefficient and unknown source function of fractional heat equations. Our approach shows that having just one set of data at an observation point ensures the existence of a weak solution for the inverse problem. Furthermore, if there is an additional datum at the observation point, it leads to a specific formula for the time-dependent source coefficient. Moreover, we investigate inverse problems involving non-local data and recovering the space-dependent source function of the fractional heat equation.     
\end{abstract} \maketitle
\tableofcontents

\section{Introduction}

Inverse problems of recovering the time-dependent coefficient and right-hand side of parabolic equations have been a point of interest for many studies, including the books \cite{1} and \cite{2}, for example. Also, vast literature is devoted to inverse problems of finding time-dependent source coefficients for heat equations. For instance, in \cite{3} several source coefficient inverse problems were considered in a cylindrical domain of $\mathbb{R}^3$ and the uniqueness and existence of the solution were established by reducing them to the system of Volterra integral equations. The existence and uniqueness results of the inverse coefficient problems and stability of the solution of inverse problems were established, for instance, in the references  \cite{5}, \cite{6},  \cite{4},  \cite{10}, \cite{11}, \cite{8}, \cite{9},    and \cite{7}.

More recently, researchers greatly focused on investigating inverse problems involving fractional operators. For instance, in \cite{26}, it was established the global uniqueness of the fractional Calderon problem with a single measurement and provided a reconstruction algorithm. Also, in \cite{27} Pu-Zhao Kow et al. considered the inverse problems of heat and wave equations involving the fractional Laplacian operator with zeroth order nonlinear perturbations and recovered nonlinear terms in the semilinear equations by employing the fractional Dirichlet-to-Neumann type map combined with the Runge approximation and the unique continuation property of the fractional Laplacian. For further discussions in this direction, we refer to \cite{30}, \cite{28}, \cite{29},  and references therein.

However, to the best of our knowledge inverse, problems identifying time and space-dependent coefficients of fractional heat equations remain open. 

The present paper aims to recover the time-dependent coefficient $p(t)$  of the following fractional Cauchy-Dirichlet problem in a bounded Lipschitz open set $\Omega\subset \mathbb{R}^n$:
\begin{equation}\label{1.1}
\begin{cases}\partial_t u(x,t)+(-\Delta)^su(x,t)=p(t)u(x,t)+f(x,t), \quad (x,t) \in \Omega \times (0,T); \\ u(x, 0)=\varphi(x), \quad x \in \Omega;\\
u(x,t)=0, \quad x\in \mathbb{R}^n
\backslash  \Omega, \quad 
t \in (0,T), 
\end{cases}    
\end{equation}
and the Cauchy problem 
\begin{eqnarray}\label{1.2}
\left\{\begin{aligned}
\partial_t u(x,t)&+(-\Delta)^su(x,t)=p(t)u(x,t)+f(x,t), \quad (x,t)\in \mathbb{R}^n \times(0,T), \\
u(x,0) & =\varphi(x), \quad x \in \mathbb{R}^n.
\end{aligned}\right.   
\end{eqnarray}
Here and after
 $(-\Delta)^s$ is the fractional Laplace operator of order $s\in (0,1)$. 

 Also, we show the existence and uniqueness of the pair of functions $(u(x,t),f(x))$ satisfying the following inverse-source problem  
\begin{eqnarray}\label{eq1.3}
\left\{\begin{aligned}
\partial_t u(x,t)&+(-\Delta)^su(x,t)=f(x), \quad (x,t) \in \Omega \times(0,T), \\
u(x, 0) & =\varphi(x), \quad x \in \Omega, \\
u(x,t)&=0, \quad x\in \mathbb{R}^n
\backslash  \Omega, \quad 
t \in (0,T), \\
u(x,T)&=\psi(x), \quad x\in \Omega.
\end{aligned}\right.   
\end{eqnarray}

 First, to find a pair $(u,p)$ in the Cauchy-Dirichlet problem \eqref{1.1}, we fix a point $q\in \Omega$ as an observation point for some time-dependent quantity and by using this we recover the time-dependent coefficient $p(t)$. It is worth noting that this method also works well to study an inverse problem for the same model but with a nonlocal additional condition. When considering the inverse problem of recovering coefficient $p(t)$ in  Cauchy problem  \eqref{1.2}, we use double data as observation points. Interestingly, this allows us to determine the unknown coefficient $p(t)$ and the solution to the problem in explicit forms. By employing the spectral theory techniques we establish the existence and uniqueness results for Problem \eqref{eq1.3}. Our techniques are partially based on the recent approach from \cite{10} in which the authors studied inverse problems for a general class of subelliptic heat equations.

We note that the fractional heat equation has applications in probability theory, more precisely from a probabilistic point of view this model represents the linear flow generated by the so-called Levy processes in stochastic
partial differential equations (see e.g. \cite{15} and \cite{16}).

The structure of our paper is as follows. Section \ref{sec2} provides basic notations, definitions, and tools for understanding and solving the problems. In Section \ref{sec3}, we use the spectral theory approach to establish the existence and uniqueness of the solution to the inverse problem \eqref{1.1} with a single datum as the observation point. In Section \ref{sec4}, we consider an inverse problem with double data and these enable us to find the time-dependent source parameter explicitly by using potential theory arguments. In Section \ref{sec5}, we apply the technique introduced in Section \ref{sec2} to study an inverse problem with a non-local datum. In Section \ref{sec6}, applying spectral theory approaches an inverse problem of identifying space-dependent source function is studied.  

\section{Preliminaries}
\label{sec2}
In this section, we give basic notations, definitions, and tools that will be used throughout the paper. 
\subsection {Fractional Sobolev spaces}\label{subsection2.1} 
Let $s\in(0,1)$ be fixed, let $\Omega$ be an open bounded subset of $\mathbb{R}^n$, $n>2s$, and let us denote by  $Q$ the set
$$Q=(\mathbb{R}^n \times\mathbb{R}^n)\backslash (\mathscr{C}\Omega\times\mathscr{C}\Omega),$$
where
$\mathscr{C}\Omega=\mathbb{R}^n\backslash \Omega$. 
The space $\mathbb{H}^s(\Omega)$ is the linear subspace of Lebegue measurable functions from $\mathbb{R}^n$ to $\mathbb{R}$ such that the restriction to $\Omega$ of any function $u$ in $\mathbb{H}^s(\Omega)$ belongs to $L^2(\Omega)$, and 
\begin{equation*}
    \text{\textit{the map}} \quad (x,y)\to (u(x)-u(y)){|x-y|^{-(n+2s)/2}} \quad \text{\textit{is in}}\quad  L^2(Q,dxdy).
\end{equation*}
 
 The norm in $\mathbb{H}^s(\Omega)$ is defined by \cite{25}
\begin{eqnarray}\label{defnorm}
    \| u\|_{\mathbb{H}^s(\Omega)}=\| u\|_{L^2(\Omega)}+\left( \int_{Q}\frac{|u(x)-u(y)|^2}{|x-y|^{n+2s}}dxdy\right)^{1/2}.
\end{eqnarray}

 The space $\mathbb{H}_0^s(\Omega)$ is defined by 
\begin{equation}
    \mathbb{H}_0^s(\Omega)=\{u\in\mathbb{H}^s(\Omega): u=0 \quad \text{a.e. in} \quad \mathscr{C}\Omega  \}
\end{equation}

 We recall some properties of the space $\mathbb{H}_0^s(\Omega)$ from \cite{25}. 

The following norm 
 \begin{equation}\label{definitionorm}
     \|u\|_{\mathbb{H}_0^s(\Omega)}=\left( \int_{\mathbb{R}^n\times\mathbb{R}^n}  \frac{|u(x)-u(y)|^2}{{|{x-y}|^{n+2s}}}dxdy \right)^{1/2}
 \end{equation}
 is equivalent to the norm defined as in \eqref{defnorm}. 
 If we take \eqref{definitionorm} as a norm on $\mathbb{H}_0^s(\Omega)$, then the following result is valid: 

 \begin{lemma}\cite{25}
     $\left(\mathbb{H}_0^s(\Omega), \|\cdot\| _{\mathbb{H}_0^s(\Omega)}\right)$ is a Hilbert space with scalar product
     \begin{equation}\label{defscalarproduct}
     \langle u,v \rangle_{\mathbb{H}_0^s({\mathbb{R}^n})}=\int_{\mathbb{R}^n\times{\mathbb{R}^n}} \frac{(u(x)-u(y))(v(x)-v(y))}{{|  {x-y} |^{n+2s}}}dxdy. \end{equation}
 \end{lemma}

\subsection{Regional fractional Laplacian and Weyl's asymptotic formula} 
 The regional fractional Laplacian operator is usually denoted by $(-\Delta)_{Reg}^s$. However, in the present paper, for simplicity, we use the notation $(-\Delta)^s$ for the regional fractional Laplacian. 

This operator acts on functions $u$ defined in $\Omega$ and extended by zero to $\mathbb{R}^n \backslash \Omega$. 
Let $u \in \mathbb{H}_0^s(\Omega)$ then the regional fractional Laplacian  is defined by (see for instance \cite{13}, \cite{14}):
$$
(-\Delta)^s u(x)=c(n,s) P.V.  \int_{\mathbb{R}^n} \frac{u(x)-u(y)}{|x-y|^{n+2s}} dy,
$$
where P.V. denotes the principle value and
$c(n,s)=\frac{\Gamma(\frac{n}{2}+s)}{|\Gamma(-s)|}\frac{4^s}{\pi^{n/2}},$ $\Gamma(\cdot)$ is Euler's gamma function.

Let $s\in(0,1)$ and let $\Omega$ be an open bounded subset of $\mathbb{R}^n$ with $n>2s$. Let us consider the following fractional Dirichlet-Laplacian spectral problem 
\begin{equation}\label{2.1}
\begin{cases} (-\Delta)^su=\lambda u, \quad  x \in \Omega;\\
u=0, \quad \text{in} \quad \mathbb{R}^n
\backslash  \Omega.  
\end{cases}    
\end{equation}

More precisely, we consider the weak formulation of \eqref{2.1}, which consists of the following eigenvalue problem
\begin{equation}\label{weakformulation}
\begin{cases} \int_{\mathbb{R}^n}\int_{\mathbb{R}^n}{(u(x)-u(x))(\varphi(x)-\varphi(y))}{|x-y|^{-(n+2s)}}dxdy=\lambda \int_{\Omega} u(x)\varphi(x)dx;\\
\quad \quad \forall \varphi(x) \in \mathbb{H}_0^s(\Omega), \quad u\in\mathbb{H}_0^s(\Omega).
\end{cases}    
\end{equation}

We recall that $\lambda \in \mathbb{R}$ is an eigenvalue of $(-\Delta)^s$ provided that there exists a nontrivial solution $u\in \mathbb{H}_0^s(\Omega)$ of problem \eqref{2.1} – in fact, of its weak formulation \eqref{weakformulation}-
and, in this case, any solution will be called an eigenfunction corresponding to the
eigenvalue $\lambda$.

It is known that the spectral problem \eqref{weakformulation} has a countable set of eigenvalues that can be written in increasing order according to their multiplicity 
\begin{eqnarray}\label{2.2}
    0<\lambda_{1}<\lambda_{2}\leq ...\leq \lambda_{k} ... \nearrow +\infty.
\end{eqnarray}

Also the sequences   $\{ \phi_{k}(x)\}_{k=1}^{+\infty}$ of the eigenfunctions corresponding to $\{\lambda_{k}\}_{k=1}^{+\infty}$ is an orthonormal basis in  $L^2(\Omega)$ and orthogonal basis in the space $\mathbb{H}_0^s(\Omega)$ (see \cite{25}, \cite{18}).

Moreover, we have the following Weyl's asymptotic formula  for the eigenvalues of the spectral problem \eqref{weakformulation} (see \cite{12} and \cite{24})
\begin{eqnarray}\label{2.3}
    \lambda_{k} \sim \frac{(2\pi)^{2s} k^{\frac{2s}{n}}}{(\omega_n  |\Omega|)^{\frac{2s}{n}}} \quad \text{as} \quad k\to +\infty, 
\end{eqnarray}
where 
$\omega_n=\pi^{n/2}\Gamma^{-1}(1+\frac{n}{2})$ is the volume of the unit ball in $\mathbb{R}^n$.

The following integration by-parts formula for regional fractional Laplacian will be also useful in our computations. 

\begin{theorem}\label{theorem2.1}\cite{19} Let $\Omega$ be a bounded Lipschitz open set in $\mathbb{R}^n$. If   $0<s<1,$  $u, v \in C_{0}^2({\Omega})$, then
\begin{eqnarray}\label{eq2.4}
\left(u,(-\Delta)^{s} v\right)_{L^2(\Omega)}=\left((-\Delta)^{s}u, v\right)_{L^2(\Omega)}.  
\end{eqnarray}
\end{theorem}

\subsection{Fourier definition of the fractional Laplacian}

\begin{definition}\cite{20}
Given $s>0$, the fractional Sobolev space is defined by
$$
\hat{H}^s\left(\mathbb{R}^n\right)=\left\{f \in L^2\left(\mathbb{R}^n\right): \int_{\mathbb{R}^n}\left(1+|\xi|^{2 s}\right)|\widehat{f}(\xi)|^2 {d} \xi<+\infty\right\},
$$
where $\widehat{f}$ denotes the Fourier transform of $f$.    
\end{definition}

We note that, the fractional Sobolev space $\hat{H}^s$ endowed with the norm
\begin{eqnarray}\label{Norm}
\|f\|_{\hat{H}^s}=\left(\int_{\mathbb{R}^n}\left(1+|\xi|^{2 s}\right)|\widehat{f}(\xi)|^2 {d} \xi\right)^{\frac{1}{2}} \quad \text { for } f \in \hat{H}^s\left(\mathbb{R}^n\right),    
\end{eqnarray}
is a Hilbert space.

\begin{definition} For $s>0,(-\Delta)^s$ denotes the fractional Laplacian defined by
\begin{eqnarray}\label{Fourierform}
(-\Delta)^s f=\mathcal{F}^{-1}\left(|\xi|^{2 s}\widehat{f}\right)   
\end{eqnarray}
for all $\xi \in \mathbb{R}^n$. \end{definition}

In other words, the fractional Laplacian $(-\Delta)^s$ can be viewed as the pseudodifferential operator with symbol $|\xi|^{2 s}$. With this definition and the Plancherel theorem, the fractional Sobolev space can be defined as:
$$
\hat{H}^s\left(\mathbb{R}^n\right)=\left\{f \in L^2\left(\mathbb{R}^n\right):(-\Delta)^{s} f \in L^2\left(\mathbb{R}^n\right)\right\},
$$
moreover, the norm
$$
\|f\|_{\hat{H}^s(\mathbb{R}^n)}=\|f\|_{L^2(\mathbb{R}^n)}+\left\|(-\Delta)^{s} f\right\|_{L^2(\mathbb{R}^n)}
$$
is equivalent to the one defined in \eqref{Norm} \cite{20}.

\subsection{Fractional heat kernel}\label{fractionalheatkernel}
The fractional heat kernel $P(x,t)$ is the fundamental solution to the fractional heat equation (see \cite{16}, \cite{17}). It is known that the integral representation of the solution of the following Cauchy problem 
\begin{equation}\label{eq2.6}
u_t(x,t)  + (-\Delta)^s u(x,t) = 0, \quad u(x,0) = u_0(x), 
\end{equation}
can be defined using the fractional heat kernel.

\begin{definition}\label{def2.4}\cite{23}
    We say that $u(x,t)$ is a strong solution of the fractional equation 
    $$u_t(x,t)  + (-\Delta)^s u(x,t) = 0, \quad (x,t) \in \mathbb{R}^n \times (0,T),$$
    if the following conditions hold:

(i) $u_t\in C(\mathbb{R}^n\times (0,T)),$

(ii) $u\in C(\mathbb{R}^n\times [0,T)),$

(iii) The equation is satisfied pointwise for every 
$(x,t) \in \mathbb{R}^n \times (0,T)$, that is,
$$u_t(x,t)+c(n,s)P.V.\int_{\mathbb{R}^n} \frac{u(x,t)-u(y,t)}{|x-y|^{n+2s}}dy=0.$$
\end{definition}

\begin{theorem}\cite{23} Let $u_0\in C(\mathbb{R}^n)\cap L^{\infty}(\mathbb{R}^n)$. Then 
$$u(x,t)=\int_{\mathbb{R}^n}P(x-y,t)u_0(y)dy$$
is a strong solution to the problem \eqref{eq2.6}, 
where $P(x,t)$ is the fractional heat kernel. 
\end{theorem}

Now we consider the following nonhomogeneous Cauchy problem 
\begin{equation}\label{eq2.7}
u_t(x,t)  + (-\Delta)^s u(x,t) = f(x,t), \quad u(x,0) = u_0(x). 
\end{equation}

Let $u_0\in C(\mathbb{R}^n)\cap L^{\infty}(\mathbb{R}^n)$ and let $f_t(x,t),$ $(-\Delta)^sf(x,t)\in L^{\infty}(\mathbb{R}^n\times (0,T))$, $f(\cdot,t)\in H^s(\mathbb{R}^n)$   and $f$ has compact support. Then, using Duhamel's principle one can show  the strong solution to the nonhomogeneous Cauchy problem \eqref{eq2.7} is represented by 
\begin{equation}\label{strongsolution}
    \begin{aligned}
u(x,t)&=\int_{\mathbb{R}^n}P(x-y,t)u_0(y)dy+\\
     &+\int_0^t\int_{\mathbb{R}^n}P(x-y,t-s)f(y,s)dyds. 
     \end{aligned}
\end{equation}

 Let $t \in[0, T]$ and that, for every $t$, or at least for a.e. $t$, the function $u(\cdot, t)$ belongs to a separable Hilbert space $V\left(\right.$ e.g. $L^2(\Omega)$ or $\left.H^1(\Omega)\right)$.
Then, one can consider $u$ as a function of the real variable $t$ with values into $V$ :
$$
u:[0, T] \rightarrow V.
$$

We recall the definitions of certain function spaces involving time that will be utilized throughout the paper.   

The set $C([0, T]; V)$ of  continuous functions $u:[0, T] \rightarrow V$, equipped with the norm
\begin{equation}\label{2.15}
\|u(\cdot,t)\|_{C([0, T] ; V)}=\max _{0 \leq t \leq T}\|u(\cdot,t)\|_V,    
\end{equation}
forms a Banach space. 

The symbol $C^1([0, T]; V)$ denotes the Banach space of functions whose derivative exists and belongs to $C([0, T]; V)$, endowed with the norm
\begin{equation}\label{2.16}
  \|u(\cdot,t)\|_{C^1([0, T] ; V)}=\|u(\cdot,t)\|_{C([0, T] ; V)}+\|{u_t(\cdot,t)}\|_{C([0, T] ; V)}.  
\end{equation}

\section{Inverse problem with single datum}
\label{sec3}

Let $\Omega$ be a bounded Lipschitz open set in $\mathbb{R}^n$. We consider the equation 
\begin{eqnarray}\label{3.1}
    \partial_tu(x,t)+(-\Delta)^su(x,t)=p(t)u(x,t)+f(x,t), \quad (x,t) \in \Omega \times (0,T),
\end{eqnarray}
with the following Cauchy-Dirichlet conditions:
\begin{eqnarray}\label{3.2}
    u(x, 0)=\varphi(x), \quad x \in \Omega,
\end{eqnarray}
\begin{eqnarray}\label{3.3}
u(x,t)=0, \quad x\in \mathbb{R}^n
\backslash  \Omega, \quad t \in (0,T),   
\end{eqnarray}
where $(-\Delta)^s$ is the regional fractional Laplacian,  $f(x,t)$ and $\varphi(x)$ are given functions such that $\varphi(\cdot), f(\cdot,t)\in L^2(\Omega).$

Assuming $p(t)\in C[0,T]$, we introduce the definition of a weak solution of the problem $\{\eqref{3.1},\eqref{3.2}, \eqref{3.3}\}$.

\begin{definition} \label{def3.1}
  A function $u\in C([0,T]; L^2(\Omega)) \cap C((0,T];\mathbb{H}_0^s(\Omega))\cap C^1((0,T]; L^2(\Omega))$ is called a weak solution of the problem $\{\eqref{3.1},\eqref{3.2},$ $\eqref{3.3}\}$ if  it satisfies   \begin{equation}\label{3.4}
      (\partial_tu,v)_{L^2(\Omega)}+     \langle u,v \rangle_{\mathbb{H}_0^s({\mathbb{R}^n})}=(p(t)u,v)_{L^2(\Omega)}+(f,v)_{L^2(\Omega)},\; \text{for all} \; v\in \mathbb{H}_0^s(\Omega),
  \end{equation}
  and the initial condition \eqref{3.2}. 
  \end{definition}
  
We consider the following inverse problem:
\begin{problem}\label{problem3.2}
    \textit{Show the existence and uniqueness of a pair of functions $(u,p)$ such that $p\in C[0,T]$, and $u$ is a weak solution to the problem $\{\eqref{3.1},\eqref{3.2},$ $\eqref{3.3} \}$ satisfying the condition}
\begin{eqnarray}\label{3.5}
    u(q,t)=w(t), \quad t \in [0,T],  
\end{eqnarray}
\textit{where $w(t)$ is a given function and $q$ is a fixed point of $\Omega$.} 
\end{problem}

First, we prove the existence and uniqueness of the weak solution of the problem $\{\eqref{3.1},\eqref{3.2},\eqref{3.3}\}$ for  arbitrary $p(t)\in C[0,T]$.

Let $\{ \phi_{k}(x)\}_{k=1}^{+\infty}$ be eigenfunctions of the fractional Dirichlet-Laplacian spectral problem \eqref{weakformulation}. We seek a weak solution to the problem $\{\eqref{3.1},\eqref{3.2},\eqref{3.3}\}$ in the form
\begin{eqnarray}\label{3.6}
    u(x,t)=\sum\limits_{k=1}^{+\infty}u_k(t)\phi_{k}(x),
\end{eqnarray}
where $(u,\phi_k)_{L^2(\Omega)}=u_k(t)$ for $k\in \mathbb{N}.$

Substituting \eqref{3.6} into \eqref{3.4} and setting $v(x)$ as $\phi_k(x)$ in \eqref{3.4}, and using the following equalities
\begin{equation*}
    \begin{aligned}
  (u_t,\phi_k)_{L^2(\Omega)}&=\frac{d}{dt}(u,\phi_k)_{L^2(\Omega)}=u^{\prime}_{k}(t), \\ \langle u,\phi_k \rangle _{\mathbb{H}_0^s(\mathbb{R}^n)}&=\lambda_{k}(u,\phi_k)_{L^2(\Omega)}=\lambda_{k}u_k(t),      
    \end{aligned}
\end{equation*}
and introducing $(f,\phi_k)_{L^2(\Omega)}=f_k(t)$ from \eqref{3.4}, we obtain 
\begin{equation}\label{3.7}
u_k^{\prime}(t)+\lambda_{k}u_k(t)=p(t)u_k(t)+f_k(t), \quad k \in \mathbb{N}, \quad t \in(0,T).
\end{equation}
Additionally, from the initial condition \eqref{3.2}, we have
\begin{equation}\label{3.8}
u_k(0)=\varphi_k, \quad k \in \mathbb{N},
\end{equation}
where $\varphi_k=(\varphi,\phi_k)_{L^2(\Omega)}.$

It is easy to verify that the solution to the problem $\{\eqref{3.7}, \eqref{3.8}\}$ has the following form 
\begin{equation}\label{coefficient}
  u_k(t)=\varphi_k e^{-\lambda_{k}t+\int_0^tp(\tau)d\tau}+\int_0^tf_k(\tau)e^{-\lambda_{k}(t-\tau)+\int_{\tau}^{t}p(y)dy}d\tau,  \quad k\in \mathbb{N}.  
\end{equation}
Substituting the obtained expression of $u_k(t)$ into \eqref{3.6}, we get  
\begin{eqnarray}\label{3.9}    
\begin{aligned}    u(x,t)&=\sum\limits_{k=1}^{+\infty}\varphi_k e^{-\lambda_{k}t+\int_0^t p(\tau)d\tau}\phi_{k}(x) \\&+\sum\limits_{k=1}^{+\infty}\left(\int_0^tf_k(\tau)e^{-\lambda_{k}(t-\tau)+\int_{\tau}^{t}p(y)dy}d\tau\right)\phi_{k}(x). 
\end{aligned}
\end{eqnarray}

The following lemma demonstrates that the function $u(x,t)$ defined by \eqref{3.9} is a weak solution to the problem $\{\eqref{3.1},\eqref{3.2}, \eqref{3.3}\}$ for any  $p(t)\in C[0,T]$.

\begin{lemma}\label{lemma3.3}
Let $\varphi \in  {L}^2(\Omega)$, 
$f \in C([0,T];L^2(\Omega))\cap C^1((0,T];L^2(\Omega))$ and  $p(t)\in C[0,T]$. Then the function $u(x,t)$ defined by \eqref{3.9} is the unique weak solution to the problem  
$\{\eqref{3.1},
\eqref{3.2},\eqref{3.3}\}$.
\end{lemma}
\begin{proof}

First, we  prove that the function $u(x,t)$ given by \eqref{3.9} belongs to the class of functions $C([0,T];L^2(\Omega))$. 

Since  $\{ \phi_{k}(x)\}_{k=1}^{+\infty}$ is an orthonormal basis in  $L^2(\Omega)$, from \eqref{3.9} we have 
\begin{equation}\label{eq3.10}
    \|u(x,t)\|^2_{L^2(\Omega)}= \sum\limits_{k=1}^{+\infty}\left\{ \varphi_k e^{-\lambda_{k}t+\int_0^t p(\tau)d\tau}+ \int_0^t  f_k(\tau)e^{-\lambda_{k}(t-\tau)+\int_{\tau}^{t}p(y)dy}d\tau \right\}^2.
\end{equation}
Hence, applying the inequality  $(a+b)^2\leq 2(a^2+b^2)$ to \eqref{eq3.10}, we obtain
\begin{eqnarray}\label{eq3.11}
\begin{aligned}    \|u(x,t)\|^2_{L^2(\Omega)} &\leq 2\sum\limits_{k=1}^{+\infty}\varphi_k^2e^{-2\lambda_{k}t+2\int_0^t p(\tau)d\tau}\\ & +  2\sum\limits_{k=1}^{+\infty}\left(\int_0^tf_k(\tau)e^{-\lambda_{k}(t-\tau)+\int_{\tau}^{t}p(y)dy}d\tau\right)^2. 
\end{aligned}
\end{eqnarray}

We show the convergence of the series on the right-hand side of  \eqref{eq3.11}. Taking $p(t)\in C[0,T]$ into account and using Parseval's identity, we have 
$$
\sum\limits_{k=1}^{+\infty}\varphi_k^2e^{-2\lambda_{k}t+2\int_0^t p(\tau)d\tau}\leq C_1 \sum\limits_{k=1}^{+\infty}\varphi_k^2=C_1 \|\varphi \|^2_{L^2(\Omega)},
$$
where $C_1$ is a constant not dependent on $k$. 

Applying the Cauchy-Schwarz inequality, it is straightforward to see 
\begin{equation} \label{3.12}
\begin{aligned}
\left\{\int_0^tf_k(\tau)e^{-\lambda_{k}(t-\tau) +\int_{\tau}^{t}p(y)dy}d\tau\right\}^2 &\leq \int_{0}^te^{-2\lambda_{k}(t-\tau) +2\int_{\tau}^{t}p(y)dy}d\tau \int_0^tf_k^2(\tau)d\tau
\\ 
& \leq C_2 \int_0^t f_k^2(\tau)d\tau.
    \end{aligned}
\end{equation}
Since $f\in C([0,T];L^2(\Omega))$, using   \eqref{3.12}, we obtain  
\begin{equation}\label{3.13}
    \begin{aligned}
        \sum\limits_{k=1}^{+\infty}\left\{\int_0^tf_k(\tau)e^{-\lambda_{k}(t-\tau) +\int_{\tau}^{t}p(y)dy}d\tau\right\}^2\leq C_2 \sum\limits_{k=1}^{+\infty}\int_0^t  f_k^2(\tau)d\tau\\ \leq C_2\|f\|_{C([0,T];L^2(\Omega)}.
    \end{aligned}
\end{equation}

Finally, summarizing the previous facts, we have 
\begin{equation}\label{continuousnorm}
  \max_{0\leq t \leq T}\|u(x,t)\|^2_{L^2(\Omega)}\leq 2  C_1 \|\varphi \|^2_{L^2(\Omega)}+2C_2\|f\|_{C([0,T];L^2(\Omega)},  
\end{equation}
from which it follows that the function $u(x, t)$ belongs to the class $C([0,T];L^2(\Omega))$ under the assumptions of Lemma  \ref{lemma3.3}.

Now, we  prove that the function defined by \eqref{3.9} belongs to the class $C((0,T];$ $\mathbb{H}_0^s(\Omega))$. To this end, we show that the series \eqref{3.9} converges uniformly in  $\mathbb{H}_0^s(\Omega)$ for any $t \in [t_0, T]$, where $0<t_0<T$. 

Rewrite \eqref{3.9} in the form 
\begin{equation}\label{3.14}
    \begin{aligned}
        u(x,t)=&\sum\limits_{k=1}^{+\infty}\sqrt{\lambda_{k}} u_k(t)\phi_k(x)/\sqrt{\lambda_k}, 
    \end{aligned}
\end{equation}
where $u_k(t)$ is defined by formula \eqref{coefficient}.

Since  $\{{\phi_k(x)/\sqrt{\lambda_{k}}}\}_{k=1}^{+\infty}$ is an orthonormal system of  $\mathbb{H}_0^s(\Omega)$,  we have 
\begin{equation}\label{eq3.16}    \|u(x,t)\|^2_{\mathbb{H}_0^s(\Omega)}=  \sum\limits_{k=1}^{+\infty} \lambda_k\left\{ \varphi_k e^{-\lambda_{k}t+\int_0^t p(\tau)d\tau}+ \int_0^t  f_k(\tau)e^{-\lambda_{k}(t-\tau)+\int_{\tau}^{t}p(y)dy}d\tau \right\}^2.
\end{equation}

Hence, applying $(a+b)^2\leq 2 a^2+2b^2$, we obtain 
\begin{equation}\label{eq3.17}
    \begin{aligned}
        \|u(x,t)\|^2_{\mathbb{H}_0^s(\Omega)} \leq &  2 \sum\limits_{k=1}^{+\infty} \left( \sqrt{\lambda_k} \varphi_k e^{-\lambda_{k}t+\int_0^t p(\tau)d\tau}\right)^2       \\+ & 2\sum\limits_{k=1}^{+\infty} \lambda_k\left(\int_0^t  f_k(\tau)e^{-\lambda_{k}(t-\tau)+\int_{\tau}^{t}p(y)dy}d\tau\right)^2.
    \end{aligned}
\end{equation}

Hence, 
we need to show the convergence of each series on the right-hand side of \eqref{eq3.17}. 

Let us consider the first series on the right-hand side of \eqref{eq3.17}. Since $p(t)\in C[0,T]$, the following estimates hold for any $t\in [t_0,T]$:
\begin{equation}\label{eq3.18}
    \begin{aligned}
  \sqrt{\lambda_{k}}e^{-\lambda_{k}t+ \int_{0}^t p(\tau)d\tau}&\leq M \sqrt{\lambda_{k}}e^{-\lambda_{k}t} =\frac{M}{t\sqrt{\lambda_{k}}}\lambda_{k}t e^{-\lambda_{k}t}\\ &\leq        \frac{M}{t\sqrt{\lambda_{k}}} \max ze^{-z}\leq \frac{M}{t_0e\sqrt{\lambda_{1}}}< +\infty,
    \end{aligned}
\end{equation}
where $M=\|e^{\int_0^tp(\tau)d\tau}\|_{C[0,T]}.$

Then, taking \eqref{eq3.18} into account, and applying Bessel's inequality, we have  
\begin{equation}\label{eq3.19}
    \sum\limits_{k=1}^{+\infty} \left( \sqrt{\lambda_k} \varphi_k e^{-\lambda_{k}t+\int_0^t p(\tau)d\tau}\right)^2 \leq C_3 \sum\limits_{k=1}^{+\infty} \varphi_k^2\leq C_3 \|\varphi\|_{L^2(\Omega)},
\end{equation}
where $C_3=\frac{M^2}{t_0^2e^2\lambda_1}$.

Now, by applying the Cauchy-Schwarz inequality, we get 
\begin{equation*} 
\begin{aligned}
\lambda_k\left\{\int_0^tf_k(\tau)e^{-\lambda_{k}(t-\tau) +\int_{\tau}^{t}p(y)dy}d\tau\right\}^2 &\leq \lambda_k\int_{0}^te^{-2\lambda_{k}(t-\tau) +2\int_{\tau}^{t}p(y)dy}d\tau \int_0^tf_k^2(\tau)d\tau
\\ 
& \leq \lambda_k M^2 \int_{0}^te^{-2\lambda_{k}(t-\tau) }d\tau \int_0^tf_k^2(\tau)d\tau\\ &
    =\lambda_k M^2 \frac{1-e^{-2\lambda_kt}}{2\lambda_k} \int_0^t f_k^2(\tau)d\tau \\ 
& \leq \lambda_k M^2 \frac{1}{2\lambda_k} \int_0^t f_k^2(\tau)d\tau \leq C_4\int_0^t f_k^2(\tau)d\tau,
    \end{aligned}
\end{equation*}
where $C_4=M^2.$

Taking the last estimate into account, we obtain 
\begin{equation}\label{eq3.20}
    \begin{aligned}
    \sum\limits_{k=1}^{+\infty} 
    \lambda_k\left\{\int_0^tf_k(\tau)e^{-\lambda_{k}(t-\tau) +\int_{\tau}^{t}p(y)dy}d\tau\right\}^2
   & \leq C_4 \sum\limits_{k=1}^{+\infty} \int_0^t f_k^2(\tau)d\tau \\ & \leq C_4\|f\|_{C((0,T];L^2(\Omega))}.        \end{aligned}
\end{equation}

From \eqref{eq3.19} and \eqref{eq3.20} it follows that 
\begin{equation}    \begin{aligned}
\label{3.16}
 \max _{t \in (0, T]}\|u(x,t) \|_{ \mathbb{H}_0^s(\Omega)}^2   
\leq  2C_3 \|\varphi\|_{L^2(\Omega)}^2+ 2C_4\|f\|_{C((0,T],L^2(\Omega))}^2.
\end{aligned}
\end{equation}

Thus, the function $u(x, t)$ belongs to  $C((0,T];\mathbb{H}_0^2(\Omega))$ under the assumptions of Lemma  \ref{lemma3.3}.

Now, we demonstrate the uniform convergence of the series obtained from \eqref{3.9} by differentiating with respect to $t$ in  $L^2(\Omega)$ for any interval $[t_0, T]$, where $0 < t_0 < T$.

By differentiating the series \eqref{3.9} with respect to $t$, we obtain
$$
\begin{aligned}
u_t(x, t) & =\sum_{k=1}^{+\infty}\left\{\left(-\lambda_k+p(t)\right) \varphi_k e^{-\lambda_k t+\int_0^t p(\tau) d \tau}+f_k(t)+\right. \\
& \left.+\int_0^t f_k(\tau)\left(-\lambda_k+p(t)\right) e^{-\lambda_k(t-\tau)+\int_\tau^t p(y) d y} d \tau\right\} \phi_k(x).
\end{aligned}
$$

Hence,  integrating by parts gives
\begin{equation}\label{3.23}
\begin{aligned}
u_t(x, t)= & \sum_{k=1}^{+\infty}\left\{\left(-\lambda_k+p(t)\right) \varphi_k e^{-\lambda_k t+\int_0^t p(\tau) d \tau}+f_k(0) e^{-\lambda_k t+\int_0^t p(y) d y}+\right. \\
& \left.+\int_0^t f_k^{\prime}(\tau) e^{-\lambda_k(t-\tau)+\int_\tau^t p(y) d y} d \tau\right\} \phi_k(x).
\end{aligned} \end{equation}

Since the system $\{{\phi_k(x)}\}_{k=1}^{+\infty}$ is orthonormal in $L^2(\Omega)$, taking $L^2$-norm of both sides of  \eqref{3.23}, we have 
\begin{equation}\label{eq3.23}
    \begin{aligned}
        \|u_t(x,t)\|^2_{L^2(\Omega)} & =\sum\limits_{k=1}^{+\infty} \left\{ \left(-\lambda_k+p(t)\right) \varphi_k e^{-\lambda_k t+\int_0^t p(\tau) d \tau} \right. \\ & + f_k(0) e^{-\lambda_k t+\int_0^t p(y) d y} +\left.\int_0^t f_k^{\prime}(\tau) e^{-\lambda_k(t-\tau)+\int_\tau^t p(y) d y} d \tau\right\}^2.
    \end{aligned}
\end{equation}

Thus, applying the inequality $(a+b+c)^2\leq 3(a^2+b^2+c^2)$ to \eqref{eq3.23}, we get 
\begin{equation}\label{eq3.24}
  \begin{aligned}
\|u_t(x,t)\|^2_{L^2(\Omega)} & \leq 3\sum\limits_{k=1}^{+\infty} \left(-\lambda_k+p(t)\right)^2 \varphi^2_k e^{-2\lambda_k t+2\int_0^t p(\tau) d \tau} 
\\ 
& + 3 \sum\limits_{k=1}^{+\infty}f_k^2(0) e^{-2\lambda_kt+2\int_0^t p(y)dy}
\\ & +3\sum\limits_{k=1}^{+\infty}\left\{\int_0^t f_k^{\prime}(\tau) e^{-\lambda_k(t-\tau)+\int_\tau^t p(y) d y} d \tau\right\}^2.
\end{aligned}
\end{equation}

Now we show convergence of both series on the right-hand side of \eqref{eq3.24}.

By applying the inequality $(a+b)^2\leq 2a^2+2b^2$ to the first series on the right-hand side of \eqref{eq3.24}, we have  
\begin{equation}\label{eq3.25}
    \begin{aligned}
 \sum\limits_{k=1}^{+\infty} \left(-\lambda_k+p(t)\right)^2  & \varphi^2_k e^{-2\lambda_k t+2\int_0^t p(\tau) d \tau}   \\ & \leq 2  \sum\limits_{k=1}^{+\infty} \lambda^2_k\varphi^2_k e^{-2\lambda_k t+2\int_0^t p(\tau) d \tau} \\ & +  2 \sum\limits_{k=1}^{+\infty} p^2(t)\varphi^2_k e^{-2\lambda_k t+2\int_0^t p(\tau) d \tau}.      
    \end{aligned}
\end{equation}

From \eqref{eq3.18}, for all $t\in[t_0,T]$, we have  
$$\lambda^2_k e^{-2\lambda_kt+2\int_0^tp(\tau)d\tau}\leq \frac{M^2}{t_0^2e^2}.$$

Taking the last estimate into account, from \eqref{eq3.25}, we obtain 
\begin{equation}
    \begin{aligned}
        \sum\limits_{k=1}^{+\infty} \left(-\lambda_k+p(t)\right)^2  & \varphi^2_k e^{-2\lambda_k t+2\int_0^t p(\tau) d \tau}   \\ & \leq 2\frac{M^2}{t_0^2e^2}  \sum\limits_{k=1}^{+\infty} \varphi^2_k +  2 M^2\|p(t)\|^2_{C[0,T]} \sum\limits_{k=1}^{+\infty} \varphi^2_k \\  & = 2M^2\left(\frac{1}{t_0^2e^2} + \|p(t)\|^2_{C[0,T]}\right) \sum\limits_{k=1}^{+\infty} \varphi^2_k \leq C_5 \|\varphi\|^2_{L^2(\Omega)}, 
    \end{aligned}
\end{equation}
where $C_5=2M^2\left(\frac{1}{t_0^2e^2} + \|p(t)\|^2_{C[0,T]}\right)$.

We also have  
$$
\sum\limits_{k=1}^{+\infty}f_k^2(0) e^{-2\lambda_kt+2\int_0^t p(y)dy}\leq M^2  \sum\limits_{k=1}^{+\infty}f_k^2(0) \leq C_6 \|f\|_{L^2(\Omega)},
$$
where $C_6=M^2.$

Applying the Cauchy-Schwarz inequality to the third series on the right-hand side of \eqref{eq3.24}, we compute 
\begin{equation*} 
\begin{aligned}
\left\{\int_0^tf^{\prime}_k(\tau)e^{-\lambda_{k}(t-\tau) +\int_{\tau}^{t}p(y)dy}d\tau\right\}^2 & \leq \int_{0}^te^{-2\lambda_{k}(t-\tau) +2\int_{\tau}^{t}p(y)dy}d\tau \int_0^t f^{\prime 2}_k(\tau)d\tau
\\ 
& \leq  M^2 \int_{0}^te^{-2\lambda_{k}(t-\tau) }d\tau \int_0^tf^{\prime 2}_k(\tau)d\tau\\ &
    =M^2 \frac{1-e^{-2\lambda_kt}}{2\lambda_k} \int_0^t f_k^2(\tau)d\tau \\ &  \leq \frac{M^2}{2\lambda_k}\int_0^t f^{\prime 2}_k(\tau)d\tau \leq \frac{M^2}{2\lambda_1}\int_0^t f^{\prime 2}_k(\tau)d\tau \\ & \leq 
    C_7 \int_0^t f^{\prime 2}_k(\tau)d\tau, 
    \end{aligned}
\end{equation*}
where $C_7=M^2/(2\lambda_1).$

Then, for the convergence of the third series on the right-hand side of \eqref{eq3.24}, we have the following estimate
\begin{equation}
    \begin{aligned}
  \sum\limits_{k=1}^{+\infty}\left\{\int_0^t f_k^{\prime}(\tau) e^{-\lambda_k(t-\tau)+\int_\tau^t p(y) d y} d \tau\right\}^2 & \leq C_7 \sum\limits_{k=1}^{+\infty}\int_0^t f^{\prime 2}_k(\tau)d\tau \\ & \leq  C_7 \|f\|_{C^1((0,T];L^2(\Omega))}.         \end{aligned}
\end{equation}

Thus, we have arrived at
\begin{equation}\label{derivativenorm}
  \max_{t\in (0,T]}\|u_t(x,t)\|^2_{L^2(\Omega)}\leq 6 C_5 \|\varphi\|^2_{L^2(\Omega)}+ 3 C_6 \|f\|_{L^2(\Omega)}+3 C_7 \|f\|_{C^1((0,T];L^2(\Omega))}.  
\end{equation}

From \eqref{continuousnorm} and \eqref{derivativenorm}, it follows that $u(x,t) \in C^1((0,T];L^2(\Omega))$.  

The proof of Lemma \ref{lemma3.3} is complete.
\end{proof}

Now we introduce the following set of functions
\begin{equation*}
D_m(\Omega)=\begin{cases}
    \varphi : \quad \varphi \in C_{0}^{2 }({\Omega}),  \; ((-\Delta)^s\varphi)^l \in C_{0}^{2}({\Omega}),\; l=0, \ldots, m-1;\\ 
\text{and} \quad ((-\Delta)^s\varphi)^m \in L^{2}({\Omega}), 
\end{cases}    
\end{equation*}
for some natural number $m$ such that $m>\frac{n}{4s}+1$,  and prove the following lemma:

\begin{lemma}\label{lemma3.4} Let $\varphi \in D_m(\Omega)$. Then we have
$$
\sum_{k=1}^{+\infty} \lambda_{k}\left|\varphi_k\right| \leq {\Lambda}^{\frac{1}{2}}\left\|((-\Delta)^s)^m \varphi\right\|_{L_2(\Omega)},
$$
where $\Lambda=\sum\limits_{k=1}^{+\infty} \frac{1}{\lambda_{k}^{2(m-1)}}$.    
\end{lemma}

\begin{proof}
By applying \eqref{eq2.4} $m$ times, one obtains 
\begin{eqnarray}\label{eq3.30}
(\varphi, \phi_k)_{L^2(\Omega)}=(-1)^m \frac{1}{\lambda_{k}^m}\left(((-\Delta)^s)^m\varphi, \phi_k \right)_{L^2(\Omega)}.  
\end{eqnarray}
Then, using the Cauchy-Schwartz and Bessel inequalities, we have
$$
\begin{aligned}
\sum_{k=1}^{+\infty} \lambda_{k}\left|\varphi_k\right|& =\sum_{k=1}^{+\infty} \frac{1}{\lambda_{k}^{m-1}}\left|\left(((-\Delta)^s)^m\varphi, \phi_k \right)_{L^2(\Omega)}\right| 
\\
& \leq \left(\sum_{k=1}^{+\infty} \frac{1}{\lambda_{k}^{2(m-1)}}\right)^{\frac{1}{2}}\left(\sum_{k=1}^{+\infty}\left|\left(((-\Delta)^s)^m\varphi, \phi_k \right)_{L^2(\Omega)}\right|^2\right)^{\frac{1}{2}}\\ & \leq {\Lambda}^{\frac{1}{2}}\left\|((-\Delta)^s)^m\varphi \right\|_{L_2(\Omega)}.
\end{aligned}
$$

Since $m>\frac{{n}}{4s}+1$,  the Weyl's asymptotic formula \eqref{2.3} implies  convergence of the series $\Lambda=\sum_{k=1}^{\infty} \frac{1}{\lambda_{k}^{2(m-1)}}$.
\end{proof}

 We choose a point $q \in \Omega$ such that $\phi_k(q)$, for all $k = 1, 2, \ldots,$ is bounded and $\phi_k(q) \neq 0$ for some $n_0 \in \mathbb{N}$. We denote by $\mathbb{N}_q$ the maximal subset of $\mathbb{N}$ such that $\phi_k(q) \neq 0$ for all $k \in \mathbb{N}_q$. We use this point as the observation point for the quantity in \eqref{3.5}.

We have the following assumptions about the given functions:

(i) $\varphi \in D_m(\Omega)$ with $\varphi_k \phi_{k}(q) \geq 0$ for $\forall k \in \mathbb{N}_q$ and $\varphi_{n_0} \phi_{n_0}(q) \ne 0$ for some $n_0 \in \mathbb{N}_q$;

(ii) $f \in C(\Omega \times[0, T])$ and $f(x, t) \in D_m(\Omega)$ with $f_k(t) \phi_{k}(q) \ne 0$ for $\forall t \in[0, T]$ and for $\forall k \in \mathbb{N}_q$;

(iii) $w \in C[0, T]$ with $w(t) \neq 0$ for $\forall t \in[0, T]$ and $w(0)=\varphi(q)$.

The existence and uniqueness of the solution to the inverse problem Problem \ref{problem3.2} is given by the following theorem:
\begin{theorem}\label{thrm3.4}  
Suppose that assumptions (i)-(iii) and the assumptions of Lemmas \ref{lemma3.3} and \ref{lemma3.4} hold. Then Problem \ref{problem3.2} has a unique solution pair $(u, p)$.
\end{theorem}

 \begin{proof}  As shown in the proof of Lemma \ref{lemma3.3}, there exists a weak solution $u(x,t)$ to the problem $\{\eqref{3.1},\eqref{3.2}, \eqref{3.3}\}$ for any $p(t)\in C[0,T]$. Now, employing the theory of integral equations and the condition \eqref{3.5}, we  recover the unknown coefficient $p(t)$. Let us introduce the notation $r(t)=e^{-\int_0^t p(\tau)d\tau}$. Then from \eqref{3.9} we obtain
$$
r(t) u(x, t)=\sum_{k=1}^{+\infty}\left(\varphi_k e^{-\lambda_{k}t}+\int_0^t f_k(\tau) e^{-\lambda_{k}(t-\tau)} r(\tau) d \tau\right) \phi_{k}(x).
$$
Hence, using the condition \eqref{3.5}, we get
$$
r(t) w(t)=\sum_{k=1}^{+\infty}\varphi_k e^{-\lambda_{k}t}\phi_{k}(q)+\sum_{k=1}^{+\infty}\left(\int_0^t f_k(\tau) e^{-\lambda_{k}(t-\tau)} r(\tau) d \tau\right) \phi_{k}(q).
$$ 
or in the case $w(t)\neq 0$
\begin{eqnarray}\label{eq3.31}
r(t)=\frac{\sum_{k=1}^{\infty} \varphi_k e^{-\lambda_{k}t} \phi_k(q)}{w(t)}+\frac{1}{w(t)} \int_0^t\left(\sum_{k=1}^{+\infty} \phi_k(q) f_k(\tau) e^{-\lambda_{k}(t-\tau)}\right) r(\tau) d \tau.    
\end{eqnarray}

Now, let us define the conditions under which the series in \eqref{eq3.31} converges. From the equality \eqref{eq3.30} for all $\varphi \in D_m(\Omega)$,  we have
$$
\left|\varphi_k\right|=\frac{1}{\lambda_{k}^m}\left|\left(((-\Delta)^s\varphi)^m, \phi_k \right)_{L^2(\Omega)} \right|.
$$

Assume that $\left|\phi_{k}(q)\right| \leq K$ for some $K= const >0$. Then the majorant of the series
$$
\sum_{k=1}^{+\infty} \varphi_k e^{-\lambda_{k} t} \phi_{k}(q)
$$
is $\sum_{k=1}^{+\infty}\left|\varphi_k\right|$, that is, by the Cauchy-Schwartz inequality we have

$$
\begin{aligned}
\sum_{k=1}^{+\infty}\left|\varphi_k\right|&=\sum_{k=1}^{+\infty} \frac{1}{\lambda_{k}^m}\left| \left(((-\Delta)^s\varphi)^m, \phi_k \right)_{L^2(\Omega)} \right| \\
& \leq\left(\sum_{k=1}^{+\infty} \frac{1}{\lambda_{k}^{2m}}\right)^{\frac{1}{2}}\left(\sum_{k=1}^{+\infty}\left| \left(((-\Delta)^s\varphi)^m, \phi_k \right)_{L^2(\Omega)} \right|^2\right)^{\frac{1}{2}}.
\end{aligned}
$$

Hence, by applying the Bessel inequality, we obtain
$$
\sum_{k=1}^{+\infty} \frac{1}{\lambda_{k}^{2m}} \sum_{k=1}^{+\infty}\left|\left(((-\Delta)^s\varphi)^m, \phi_k \right)_{L^2(\Omega)}\right|^2 \leq \sum_{k=1}^{+\infty} \frac{1}{\lambda_{k}^{2m}}\left\|((-\Delta)^s\varphi)^m\right\|_{L_2(\Omega)}^2.
$$

 Weyl's asymptotic formula \eqref{2.3} yields that the first series in \eqref{eq3.31} is convergent. Similarly, one can show the convergence of the second series in \eqref{eq3.31}. Considering \eqref{eq3.31} as a second-kind Volterra integral equation for $r(t)$, we see that $r(t)$ is a unique solution  and continuous in $[0, T].$ We have
$$
r(t)=e^{-\int_0^t p(\tau) d \tau} \Longrightarrow p(t)=-\frac{r^{\prime}(t)}{r(t)} .
$$

The function $r(t)$ is continuously differentiable if the series $\sum_{k=1}^{+\infty} \varphi_k \lambda_{k} e^{-\lambda_{k} t} \phi_k(q)$ is uniformly convergent. This is indeed the case since the majorant series $M \sum_{k=1}^{+\infty} \lambda_{k}\left|\varphi_k\right|$ converges for $m>\frac{n}{4s}+1$ by Lemma \ref{lemma3.4}.
 \end{proof}

We demonstrate  an explicit example that satisfies the assumptions of Theorem \ref{thrm3.4}. First, let us recall some known facts about the one-dimensional fractional Dirichlet-Laplacian spectral problem:
\begin{eqnarray}\label{eq3.32}
(-\Delta)^s\phi(x) = \lambda \phi(x), \quad x \in (-1, 1),
\end{eqnarray}
where $\phi(x) \in L^2(-1, 1)$ is extended to $\mathbb{R}$ by $0$, and $\left(-\Delta\right)^{s}$ is the one-dimensional regional fractional Laplacian defined as
$$(-\Delta)^s g(x)=c_\alpha P.V.\int\limits_{-\infty}^{+\infty}\frac{g(x)-g(y)}{|x-y|^{1+2s}}dy, \quad 0<s<1, \quad x \in \mathbb{R},
$$
with $c_{\alpha}=\frac{1}{\pi}\Gamma(1+\alpha)\sin \frac{\alpha \pi}{2}.$

It is known that the spectral problem \eqref{eq3.32} has a countable set of eigenvalues
$$0<\lambda_1<\lambda_2\leq \lambda_3\leq ...\leq \lambda_k \leq ..., \nearrow +\infty $$
and the corresponding eigenfunctions $\{\phi_k(x)\}_{k=1}^{+\infty}$ form a complete orthonormal set in $L^2(-1,1)$. Additionally, the eigenfunctions $\phi_k(x)$ are uniformly bounded for $k \geq 1$ and $x \in (-1,1)$, i.e., $|\varphi_k(x)| \leq 3$ for all $k \geq 1$ and $x \in (-1,1)$ (see \cite{21}).

\begin{example}\label{example3.7} The following  inverse problem
\begin{eqnarray}\label{eq3.33}
\begin{aligned}
    \partial_tu(x,t)+\left(-\Delta\right)^{s}u(x,t)&=p(t)u(x,t)+e^{-\lambda_1 t}\phi_1(x), \quad (x,t)\in(-1,1) \times (0,T),\\
    u(x, 0)&=\phi_1(x), \quad x \in (-1,1),\\
u(x,t)&=0, \quad \text{in} \quad \mathbb{R}
\backslash  (-1,1),  \\ 
u(q,t)&= e^{-\lambda_1t}\phi_1(q), \quad t \in [0,T],
\end{aligned}
\end{eqnarray}
satisfies all the assumptions of Theorem \eqref{thrm3.4}.

Indeed, since $\phi_1(x)$ is an eigenfunction of the spectral problem \eqref{eq3.32}, we have $\varphi \in D_2(-1,1)$ and also $f(\cdot, t) \in D_{2}(-1,1)$ for all  $t \in [0,T]$. Furthermore  $\varphi_k=(\varphi, \phi_k)=0,$  $k>1$, $\varphi_1=1,$ $f_k(t)=(f,\phi_k)=0,$ $k>1$, $f_1(t)=e^{-\lambda_1 t}$.  Also $\varphi_1 \phi_1(q)=\phi_1(q) \ne 0,$  and $f_1(t)\phi_1(q)=e^{-\lambda_1 t}\phi_1(q) \ne 0$ with $w(0)=\phi_1(q)$.
Thus, all the assumptions of Theorem \ref{thrm3.4} are satisfied. We have the following integral equation for determining the unknown function $r(t)$ 
$$r(t)-\int_0^t e^{\lambda_1(\tau-t)}r(\tau)d\tau=1, \quad t\in [0,T].$$

This integral equation has the unique solution in the form
$$r(t)=\frac{1}{\lambda_1-1}(\lambda_1e^{\lambda_1t}-e^t).$$ 

Therefore, the solution to the inverse problem \eqref{eq3.33} is given by 
$$u(x,t)=\left(e^{-\lambda_1t+\int_0^tp(\tau)d\tau}+\int_{0}^{t}e^{-\lambda_1t+\int_{\tau}^tp(y)dy}d\tau\right)\phi_1(x), 
\quad p(t)=-\frac{{\lambda}_1^2e^{\lambda_1 t}-e^t}{{\lambda}_1 e^{\lambda_1t}-e^t}.$$

\end{example}

\section{Inverse problem with double datum}
\label{sec4}

Consider the equation  
\begin{eqnarray}\label{eq4.1}
\partial_t u(x,t)+(-\Delta)^su(x,t)=p(t)u(x,t)+f(x,t), \quad (x,t)\in \mathbb{R}^n \times(0,T), 
\end{eqnarray}
with the following initial condition 
\begin{eqnarray}\label{eq4.2}
u(x, 0)  =\varphi(x), \quad x \in \mathbb{R}^n,
\end{eqnarray}
where $(-\Delta)^s$ is the fractional Laplacian operator defined by \eqref{Fourierform}, $\varphi$ and $f$ are given functions. Assume that there exists $q \in \mathbb{R}^n$ such that $f(q, t) \in C^1(0, T)$ is a continuously differentiable function with $f(q, t) \neq 0$. Now, we fix $q \in \mathbb{R}^n$ as an observation point for the following two time-dependent quantities:
\begin{eqnarray}\label{eq4.3}
w_1(t):=v_1(q,t), \quad w_2(t):=v_2(q, t), \quad t \in [0,T].    
\end{eqnarray}

Here
\begin{eqnarray}
\left\{\begin{aligned}\label{eq4.4}
\partial_t v_1(x, t) & +(-\Delta)^s v_1(x, t)=p(t) v_1(x, t)+f(x, t), \quad (x,t)\in \mathbb{R}^n \times(0, T), \\
v_1(x, 0) & =\varphi(x), \quad x \in \mathbb{R}^n,
\end{aligned}\right.    
\end{eqnarray}
and
\begin{eqnarray}
\left\{\begin{aligned}\label{eq4.5}
\partial_t v_2(x, t) & +(-\Delta)^s v_2(x, t)=p(t) v_2(x, t)+(-\Delta)^s f(x, t), \ (x,t) \in \mathbb{R}^n \times(0, T), \\
v_2(x, 0) & =(-\Delta)^s \varphi(x), \quad x \in \mathbb{R}^n.
\end{aligned}\right.    
\end{eqnarray}

We consider the following inverse problem for equation \eqref{eq4.1}:

\begin{problem}\label{problem4.1}
    Find a strong solution $u(x,t)$ and a function $p(t)\in C[0,T]$ in an explicit form satisfying  \eqref{eq4.1}-\eqref{eq4.2} under the condition \eqref{eq4.3}.
\end{problem}

\begin{theorem}\label{Theorem4.3}
Let $\varphi, (-\Delta)^s \varphi \in C(\mathbb{R}^n)\cap L^{\infty}(\mathbb{R}^n)$, $f_t(x,t),$ $(-\Delta)^sf(x,t)\in L^{\infty}(\mathbb{R}^n\times (0,T))$ and  $\partial_t((-\Delta)^sf(x,t)),$  $(-\Delta)^{2s}f(x,t)\in L^{\infty}(\mathbb{R}^n\times (0,T))$. Suppose $(-\Delta)^sf(\cdot,t)\in \hat H^s(\mathbb{R}^n)$ for all $t\in[0,T]$  and the functions $f$ and $(-\Delta)^sf$ have compact support. Let $q \in \mathbb{R}^n$ be such that $f(q, t) \in C^1[0, T]$ is a continuously differentiable function with $f(q, t) \neq 0$. Let $w_1$ and $w_2$ defined in \eqref{eq4.3} be the observation data. Then there exists a unique solution pair $(u, p)$ for the inverse problem \eqref{eq4.1}-\eqref{eq4.3} with
$$
p(t) w_1(t)=-w_1^{\prime}(t)+w_2(t)-f(q, t).
$$    
\end{theorem}

\begin{proof}

Note that for arbitrary continuous function $p(t)$, the solution $u$ to the direct problem, using \eqref{strongsolution}, can be defined by:
$$
\begin{array}{r}
\begin{aligned}
 u(x, t)=&\exp \left(\int_0^t p(\tau) d \tau\right) \left\{  \int_{\mathbb{R}^n} P(x-y, t) \varphi(y) dy \right. \\ 
+&\left.\int_0^t \int_{\mathbb{R}^n} P(x-y, t-\tau) \exp \left(-\int_0^\tau p(s) d s\right) f(y, \tau) dy d\tau  \right\}.  
\end{aligned}
\end{array}
$$

Set 
\begin{equation}\label{vtourtop}
v(x, t)=r(t) u(x, t), \quad r(t)=\exp \left(-\int_0^t p(\tau) d \tau\right),
\end{equation}
that is,
$$
p(t)=-\frac{r^{\prime}(t)}{r(t)}, \quad u(x, t)=\frac{v(x, t)}{r(t)} .
$$

Thus, we have
$$
\left\{\begin{aligned}
\partial_t v(x, t) & +(-\Delta)^s v(x, t)=r(t) f(x, t), \quad (x,t) \in \mathbb{R}^n \times(0, T), \\
v(x, 0) & =\varphi(x), \quad x \in \mathbb{R}^n.
\end{aligned}\right.
$$

We now fix a point $q \in \Omega$ as an observation point for two time-dependent quantities:
$$
\tilde{w}_1(t):=w_1(t, q), \quad \tilde{w}_2(t):=w_2(t, q), \quad t \in [0,T].
$$

Here
$$
\left\{\begin{aligned}
\partial_t w_1(x, t) & +(-\Delta)^s w_1(x, t)=r(t) f(x, t), \quad (x,t) \in \mathbb{R}^n \times(0, T), \\
w_1(x, 0) & =\varphi(x), \quad x \in \mathbb{R}^n,
\end{aligned}\right.
$$
and
$$
\left\{\begin{aligned}
\partial_t w_2(x, t) &+ (-\Delta)^s w_2(x, t)=r(t)(-\Delta)^s f(x, t), \quad (x,t) \in \mathbb{R}^n \times(0, T), \\
w_2(x, 0) & =(-\Delta)^s \varphi(x), \quad x \in \mathbb{R}^n.
\end{aligned}\right.
$$

Under the assumptions of Theorem \ref{Theorem4.3}, the strong solutions of the Cauchy problems \eqref{eq4.4} and \eqref{eq4.5} can be presented as follows:
\begin{equation*}
    \begin{split}
w_1(x, t)& =\int_{\mathbb{R}^n} P(x-y, t) \varphi(y) d y\\ &+\int_0^t \int_{\mathbb{R}^n} P(x-y, t-\tau) r(\tau) f(y, \tau) dy d\tau        
    \end{split}
\end{equation*}
and
\begin{equation*}
    \begin{split}
w_2(x,t)&=\int_{\mathbb{R}^n} P(x-y, t)(-\Delta)^s \varphi(y) dy\\ &+\int_0^t \int_{\mathbb{R}^n} P(x-y, t-\tau) r(\tau) (-\Delta)^s f(y,\tau) dy d\tau.        
    \end{split}
\end{equation*}

Then, for $q \in \mathbb{R}^n$ and for all $t \in[0,T]$, we have
\begin{equation*}
\begin{split}
 w_1(q, t)& =\int_{\mathbb{R}^n} P(q-y, t) \varphi(y) dy\\ & +\int_0^t \int_{\mathbb{R}^n} P(q-y, t-\tau) r(\tau) f(y, \tau) d y d \tau= \tilde{w}_1(t), \\
 w_2(q, t)& =\int_{\mathbb{R}^n} P(q-y, t) (-\Delta)^s \varphi(y) d y \\ & +\int_0^t \int_{\mathbb{R}^n} P(q-y, t-\tau) r(\tau) (-\Delta)^s f(y, \tau) dy d\tau=\tilde{w}_2(t).
\end{split}   
\end{equation*}

By differentiating $\tilde{w}_1$ and applying the Leibniz integral rule, then using 
\begin{eqnarray*}
    \begin{aligned}
  \partial_t P(x-y, t-\tau)=&-(-\Delta)^s P(x-y, t-\tau), \quad t>\tau, \\  \partial_t P(x-y, t)=&-(-\Delta)^s P(x-y, t), \quad t>0,          \end{aligned}
\end{eqnarray*}
and also using the symmetry of the operator $(-\Delta)^s$, we get
$$
\begin{aligned}
 \tilde{w}_1^{\prime}(t)&=\int_0^t \int_{\mathbb{R}^n} \partial_t P(q-y, t-\tau) r(\tau) f(y, \tau) dy d\tau \\
& +r(t) f(q,t)+\int_{\mathbb{R}^n} \partial_t P(q-y,t) \varphi(y) dy
\\&=r(t) f(q,t)-\int_{\mathbb{R}^n} (-\Delta)^s P(q-y,t) \varphi(y) d y \\
&  -\int_0^t \int_{\mathbb{R}^n} (-\Delta)^s P(q-y,t-\tau) r(\tau) f(y, \tau) dy d\tau\\
& =-\int_0^t \int_{\mathbb{R}^n} P(q-y, t-\tau) r(\tau) (-\Delta)^sf(y, \tau) dyd\tau\\
&+r(t)f(q,t) -\int_{\mathbb{R}^n} P(q-y, t) (-\Delta)^s \varphi(y) dy\\ & =-\tilde{w}_2(t)+r(t) f(q, t) .
\end{aligned}
$$

Hence, we obtain
\begin{eqnarray}\label{eq4.6}
\tilde{w}_2(t)=-\tilde{w}_1^{\prime}(t)+r(t) f(q, t).    
\end{eqnarray}

Moreover, we have the following relations
$$
\tilde{w}_1(t):=r(t) w_1(t), \quad \tilde{w}_2(t):=r(t) w_2(t), \quad t \in [0,T] .
$$

Taking these into account, from \eqref{eq4.5}, we get 
$$
p(t) w_1(t)=-w_1^{\prime}(t)+w_2(t)-f(q, t).
$$

We note that the direct Cauchy problems \eqref{eq4.1}, \eqref{eq4.4}, and  \eqref{eq4.5} have a unique solution under the assumptions of Theorem \ref{Theorem4.3}. It implies that there exists a unique pair $(v, r)$, that is, $(u, p)$ for the inverse problem \ref{problem4.1}. 

The proof of Theorem \ref{Theorem4.3} is complete.
\end{proof}

\section{Inverse problem with non-local datum}
\label{sec5}

Let $\Omega$ be a bounded Lipschitz open set in $\mathbb{R}^n$. We consider the following equation:
\begin{equation}\label{5.1}
\partial_tu(x,t) + (-\Delta)^s u(x,t) = r(t)f(x,t), \quad (x,t) \in \Omega \times (0,T),
\end{equation}
with Cauchy-Dirichlet conditions \eqref{3.2} and \eqref{3.3}, where $(-\Delta)^s$ is the regional fractional Laplacian operator, and $\varphi(x)$ and $f(x,t)$ are given functions such that $\varphi,$ $f(\cdot,t) \in L^2(\Omega)$, and $r(t)$ is an unknown coefficient. 

Note that the equation  \eqref{5.1} is equivalent to \eqref{3.1} through the transform \eqref{vtourtop}.

For any $r(t)\in C[0,T]$, the definition of a weak solution to $\{\eqref{5.1},\eqref{3.2},\eqref{3.3}\}$ is given as follows:
\begin{definition} \label{def5.1}
  A function $u\in C([0,T]; L^2(\Omega)) \cap C((0,T];\mathbb{H}_0^s(\Omega))\cap C^1((0,T]; L^2(\Omega))$ is called a weak solution of the problem $\{\eqref{5.1},\eqref{3.2},$ $\eqref{3.3}\}$ if  it satisfies  \begin{equation}\label{5.2}
      (\partial_tu,v)_{L^2(\Omega)}+     \langle u,v \rangle_{\mathbb{H}_0^s({\mathbb{R}^n})}=(r(t)f,v)_{L^2(\Omega)},\; 
      \text{for all} \; v\in \mathbb{H}_0^s(\Omega),
  \end{equation}
  and the initial condition \eqref{3.2}. 
  \end{definition}

\begin{problem}\label{problem5.2}
Show the existence and uniqueness of a pair of functions $(u(x,t),r(t))$ such that $r(t)\in C[0,T]$, and $u(x,t)$ is a weak solution to the problem $\left\{\eqref{5.1}, \eqref{3.2},\right.$ $\left. \eqref{3.3}\right\}$, satisfying the nonlocal condition
\begin{eqnarray}\label{5.3}
\int_{\Omega} \omega(x) u(x, t) d x=w(t), \quad t \in[0, T],    
\end{eqnarray}
where $\omega(x)$ and $w(t)$ are given functions such that $\omega(x) \in L_2(\Omega)$ and $w \in C^1[0, T]$.     
\end{problem}

Let $\{ \phi_{k}(x)\}_{k=1}^{+\infty}$ 
be eigenfunctions of the fractional Dirichlet-Laplacian spectral problem \eqref{weakformulation}. We seek a solution to the problem $\{\eqref{5.1},\eqref{3.2},\eqref{3.3}\}$ in the form \eqref{3.6}. Substituting \eqref{3.6} into \eqref{5.2} and setting $v(x)$ as $\phi_k(x)$ in \eqref{5.2}, and performing similar steps as in the case of Problem \ref{problem3.2}, we obtain:
\begin{equation}\label{Cauchyproblem}
\begin{cases}
u_k^{\prime}(t) + \lambda_{k}u_k(t) = r(t)f_k(t), \quad t \in (0, T), \\
u_k(0) = \varphi_k, 
\end{cases}
\end{equation}
where $\varphi_k=(\varphi,\phi_k)_{L^2(\Omega)}$ and $f_k(t)=(f,\phi_k)_{L^2(\Omega)}$.

The solution of Cauchy problem \eqref{Cauchyproblem} has the following form 
$$u_k(t)=\varphi_k e^{-\lambda_{k}t}+\int_0^tf_k(\tau)e^{-\lambda_{k}(t-\tau)}d\tau, \quad k=1,2,\ldots .$$

Substituting the obtained expression of $u_k(t)$ into \eqref{3.6}, we get  
\begin{eqnarray}\label{5.5}
u(x,t)=\sum_{k=1}^{\infty}\left(\varphi_k e^{-\lambda_{k} t}\phi_k(x) +\int_0^t f_k(\tau) e^{-\lambda_{k}(t-\tau)} r(\tau) d \tau \right) \phi_k(x).    
\end{eqnarray}

The following lemma shows that the function $u(x,t)$ defined by \eqref{5.5} is a weak solution to the problem $\{\eqref{5.1},\eqref{3.2},$ $\eqref{3.3}\}$ for any function $r(t)\in C[0,T]$.
\begin{lemma}\label{lemma5.3}
Let $\varphi(x)\in  L^2(\Omega)$ and $f(x,t)\in C([0,T];L^2(\Omega) )\cap C^1((0,T];L^2(\Omega))$. Then the function $u(x,t)$ defined by \eqref{5.5} is a weak solution of the problem  $\{\eqref{5.1},\eqref{3.2},$ $\eqref{3.3}\}$ for any arbitrary $r(t)\in C[0,T]$.
\end{lemma}
\begin{proof}
The proof is similarly to the one of Lemma \ref{lemma3.3}. 
\end{proof}

\begin{theorem}\label{thrm5.4}
Let the assumptions of Lemma \ref{lemma5.3}  and the following conditions hold:

1. $\varphi \in D_m(\Omega)$;

2. $f \in C(\Omega \times[0, T]), f(x, t) \in D_m(\Omega)$ and $\int_{\Omega} \omega(x) f(x, t) d x \neq 0$ for $\forall t \in[0, T]$;

3. $w \in C^1[0, T]$ with $w(t) \neq 0$ for $\forall t \in[0, T]$ and $w(0)=\int_{\Omega} \omega(x) \varphi(x) d x$.

Then, there exists a unique pair of functions $(u, r)$ that solves Problem \ref{problem5.2}.
\end{theorem}

\begin{proof}

From  Lemma \ref{lemma5.3} it follows that the function $u(x,t)$ defined by  \eqref{5.5} is a weak solution of the direct problem for any $r(t)\in C[0,T]$.
Now, we show the existence and uniqueness  of the solution of the inverse problem \ref{problem5.2}. 

To this end, differentiating  \eqref{5.5} with respect to  $t$, we obtain 
\begin{equation*}
    \begin{split}
u_t(x, t)&=\sum_{k=1}^{+\infty}\left(-\lambda_{k} \varphi_k e^{-\lambda_{k} t}\right)\phi_k(x)\\ &-\sum_{k=1}^{+\infty}\left(\lambda_k \int_0^t f_k(\tau) e^{-\lambda_{k}(t-\tau)} r(\tau) d \tau-f_k(t) r(t)\right) \phi_k(x).       
    \end{split}
\end{equation*}
Hence, from the over-determination condition \eqref{5.3}, we have
\begin{equation*}
    \begin{split}
\int_{\Omega} \omega(x) u_t(x,t)dx &=\sum_{k=1}^{+\infty}  \left(-\lambda_{k} \varphi_k e^{-\lambda_{k} t}\right)\int_{\Omega} \omega(x) \phi_k(x) d x \\ &+\sum_{k=1}^{+\infty}\left(-\lambda_{k} \int_0^t f_k(\tau) e^{-\lambda_{k}(t-\tau)} r(\tau) d \tau\right) \int_{\Omega} \omega(x) \phi_k(x) d x \\
& +r(t) \int_{\Omega} \omega(x) \sum_{k=1}^{+\infty} f_k(t) \phi_k(x) dx=w^{\prime}(t)      
    \end{split}
\end{equation*}
or
\begin{eqnarray}\label{eq5.4}
    \begin{aligned}
r(t)=&\frac{w^{\prime}(t)+\sum_{k=1}^{\infty} \lambda_{k} \varphi_k e^{-\lambda_{k} t} \int_{\Omega} \omega(x) \phi_k(x) dx}{\int_{\Omega} \omega(x) f(x,t) dx}\\
+\frac{1}{\int_{\Omega} \omega(x) f(x, t)dx} &\int_0^t\left(\sum_{k=1}^{+\infty} \lambda_{k} f_k(\tau) e^{-\lambda_{k}(t-\tau)} d \tau \int_{\Omega} \omega(x) \phi_k(x) dx\right) r(\tau) d\tau 
\end{aligned}
\end{eqnarray}
since 
$$\int_{\Omega} \omega(x) \sum\limits_{k=1}^{+\infty} f_k(t) \phi_k(x) dx=\int_{\Omega} \omega(x) f(x,t) dx.$$

We have $\omega \in L^2(\Omega)$, that is, $\int_{\Omega} \omega^2(x) d x \leq M$ in \eqref{eq5.4} for some $M=$ const $>0$. The series
$$
\sum_{k=1}^{\infty} \lambda_{k} \varphi_k e^{-\lambda_{k} t} \int_{\Omega} \omega(x) \phi_k(x) dx
$$
is uniformly convergent since the majorant series $M \sum_{k=1}^{\infty} \lambda_{k}\left|\varphi_k\right|$ is convergent by Lemma \ref{lemma3.4}. Thus, the Volterra integral equation \eqref{eq5.4} has a unique continuous solution.
    The uniqueness of the weak solution to the problem $\{\eqref{5.1},\eqref{3.2},\eqref{3.3}\}$ for any $r(t)\in C[0,T]$ follows from the completeness of the system of eigenfunctions $\{\phi_k\}_{k\in\mathbb{N}}$ in $L^2(\Omega)$.
\end{proof}

Here we provide an example that satisfies all the assumptions of Theorem \ref{thrm5.4}.

\begin{example}\label{example5.6} The following one-dimensional inverse problem
\begin{eqnarray}\label{5.7}
\begin{aligned}
    \partial_tu(x,t)+\left(-\Delta\right)^{s}u(x,t)&=r(t)(1+\sin^2t)\phi_1(x), \quad (x,t)\in(-1,1) \times (0,T),\\
    u(x, 0)&=\phi_1(x), \quad x \in (-1,1),\\
u(x,t)&=0, \quad \text{in} \quad \mathbb{R}
\backslash  (-1,1),  \\ 
\int_{\Omega} \phi_1(x) & u(x, t) dx=1+t^2, \quad t \in[0, T],    \end{aligned}
\end{eqnarray}
satisfies all the assumptions of Theorem \ref{thrm5.4}, where $\phi_1(x)$ is the first eigenfunction of one-dimensional fractional problem \eqref{eq3.32}.

Indeed, since $\varphi(x)=\phi_1(x),$ $f(x,t)=(1+\sin^2t)\phi_1(x)$, and $\omega(x)=\phi_1(x)$, $w(t)=1+t^2$, we have $\varphi(x)\in D_2(-1,1),$ and $f(x,t)\in D_2(-1,1)$ for all $t\in[0,T]$. Furthermore, $\omega(x)\in L^2(-1,1)$ since $\phi_1(x)\in L^2(-1,1)$ and $w(0)=1=\int_{-1}^1\phi_1^2(x)dx=1$ from the orthonormality of eigenfunctions $\phi_k(x),$ $k=1,2,\ldots $. 

Also, $w(t)\in C^1[0,T]$ and 
\begin{equation*}
    \begin{split}  \int_{-1}^1\omega(x)f(x,t)dx&=\int_{-1}^1\phi^2_1(x)(1+\sin^2t)dx\\&=(1+\sin^2t)\int_{-1}^1\phi^2_1(x)dx=1+\sin^2t\ne0      
    \end{split}
\end{equation*}
for all $t\in[0,T]$. Thus, all the assumptions of Theorem \ref{thrm5.4} are satisfied. To find the unknown coefficient $r(t)$, we get the following integral equation:
\begin{equation}\label{5.8}
    r(t)=\frac{2t+\lambda_1e^{-\lambda_1t}}{1+\sin^2t}+\frac{\lambda_1}{1+\sin^2t}\int_0^t(1+\sin^2\tau)e^{-\lambda_1(t-\tau)}r(\tau)d\tau, \quad t\in[0,T],
\end{equation}
where $\lambda_1$ is the first eigenvalue of fractional spectral problem \eqref{eq3.32}.

Introducing the notation 
    $$r_1(t)=(1+\sin^2t)e^{\lambda_1t}r(t),$$
from  \eqref{5.8} we obtain
$$r_1(t)-\lambda_1\int_0^tr_1(\tau)d\tau=2te^{\lambda_1t}+\lambda_1.$$ 
It implies 
$$r_1^{\prime}(t)-\lambda_1r_1(t)=2\lambda_1te^{\lambda_1t}+2e^{\lambda_1t}.$$

The solution of this ordinary differential equation, satisfying the condition $r_1(0)=\lambda_1$, has the form    $$r_1(t)=e^{\lambda_1t}(\lambda_1+2t+\lambda_1t^2).$$

Then, the solution of inverse problem \eqref{5.7}, i.e. the pair of functions $r(t)$ and $u(x,t)$ are defined by 
$$r(t)=\frac{\lambda_1+2t+\lambda_1t^2}{1+\sin^2t},$$
and
$$u(x,t)=\left(e^{-\lambda_1t}+\int_0^t(1+\sin^2\tau)e^{-\lambda_1(t-\tau)}r(\tau)d\tau\right)\phi_1(x).$$

\end{example}

\section{Inverse problem identifying space variable source function}
\label{sec6}

Let $\Omega$ be a bounded Lipschitz open set in $\mathbb{R}^n$.  We consider the fractional heat equation
\begin{eqnarray}\label{6.1}
    \partial_t u(x,t)+(-\Delta)^su(x,t)=f(x),\quad (x,t)\in \Omega \times (0,T), 
\end{eqnarray}
with the  Cauchy-Dirichlet conditions \eqref{3.2} and \eqref{3.3}. 

First, we introduce the definition of a weak solution to the direct problem $\left\{\eqref{6.1},\right.$ $ \left.\eqref{3.2},\eqref{3.3}\right\}$.

\begin{definition} \label{def6.1}
  A function $u\in C([0,T]; L^2(\Omega)) \cap C((0,T];\mathbb{H}_0^s(\Omega))\cap C^1((0,T]; L^2(\Omega))$ is called a weak solution of the problem $\{\eqref{6.1},\eqref{3.2},$ $\eqref{3.3}\}$ if  it satisfies \begin{equation}\label{6.4}
 (\partial_tu,v)_{L^2(\Omega)}+     \langle u,v \rangle_{\mathbb{H}_0^s({\mathbb{R}^n})}=(f,v)_{L^2(\Omega)}\; \text{for all}\; v\in \mathbb{H}_0^s(\Omega),
  \end{equation}
  and the initial condition \eqref{3.2}. 
  \end{definition}

\begin{problem}\label{problem6.2}
Show the existence and uniqueness of a pair of functions $\{u(x,t),f(x)\}$ such that $f(x)\in L^2(\Omega)$, and $u(x,t)$ is a weak solution to the problem $\left\{\eqref{6.1}, \eqref{3.2},\right.$ $\left. \eqref{3.3}\right\}$, satisfying the condition
\begin{eqnarray}\label{6.5}
    u(x,T)=\psi(x), \quad x \in \Omega,
\end{eqnarray}
where $\psi(x)$ is a given function.    
\end{problem}

\begin{theorem}\label{Theorem6.3}
 Let $\varphi$, $\psi \in \mathbb{H}_0^s(\Omega)$. Then the following unique pair of functions $\{u(x,t),f(x)\}$ solves Problem \ref{problem6.2}:
 \begin{eqnarray}\label{6.6}
u(x,t)=\sum\limits_{k=1}^{+\infty}\left(\varphi_k+\frac{\varphi_k-\psi_k}{1-e^{-\lambda_{k}T}}\right)(e^{-\lambda_{k}t}-1)\phi_{k}(x),
 \end{eqnarray}
 and \begin{eqnarray}\label{6.7}
    f(x)=\sum_{k=1}^{+\infty}  \lambda_k\frac{\psi_k-\varphi_ke^{-\lambda_kT}}{\left(1-e^{-\lambda_{k} T}\right)}\phi_k(x).
 \end{eqnarray}
\end{theorem}

\begin{proof}

First, we prove the existence of the solution. Let $\{ \phi_{k}(x)\}_{k \in \mathbb{N}}$ 
be eigenfunctions of the fractional Dirichlet-Laplacian spectral problem \eqref{weakformulation}.  We seek  $u(x,t)$ and $f(x)$ in the forms
\begin{eqnarray}\label{6.8}
    u(x,t)=\sum\limits_{k=1}^{+\infty}u_k(t)\phi_{k}(x), \quad f(x)=\sum\limits_{k=1}^{+\infty}f_k \phi_{k}(x). 
\end{eqnarray}

Substituting \eqref{6.8} into \eqref{6.4} and taking $v$ as $\phi_k(x)$, we get the following equation for the functions $u_k(t)$ 
and $f_k$:
$$
u_k^{\prime}(t)+\lambda_{k}u_k(t)=f_k, \quad k\in \mathbb{N}.
$$
It has the solution
\begin{eqnarray}\label{6.9}
u_k(t)=\frac{f_k}{\lambda_{k}}+C_ke^{-\lambda_{k}t},  \quad k\in \mathbb{N},
\end{eqnarray}
where the coefficient $C_k$ and $f_k$ are unknown. 

Using conditions \eqref{6.5} and \eqref{6.4}, we find these coefficients. Let $\varphi_k$ and $\psi_k$ be the coefficients of the expansions of the functions $\varphi(x)$ and $\psi(x)$, i.e.
$$\varphi_k=(\varphi, \phi_{k})_{L^2(\Omega)}, \quad \psi_k=(\psi, 
\phi_{k})_{L^2(\Omega)}.$$
Then, we have 
$$u_k(0)=\varphi_k, \quad  u_k(T)=\psi_k.$$
Using these conditions, we find $C_k$ and $f_k$ as follows 
$$C_k=\frac{\varphi_k-\psi_k}{1-e^{-\lambda_{k}T}},  \quad f_k=\lambda_{k}\varphi_k-\lambda_{k}C_k.$$

Substituting obtained expressions of $u_k(t)$ and $f_k$ into \eqref{6.8}, we obtain \eqref{6.6} and \eqref{6.7}.

First, we show that the function $u(x, t)$ defined by \eqref{6.6} belongs to the class $C\left([0, T]; L^2(\Omega)\right)$ under the assumptions of Theorem \ref{Theorem6.3}. Since $\left\{\phi_k(x)\right\}_{k=1}^{+\infty}$ is an orthonormal system in $L^2(\Omega)$, we have
$$\|u(x, t)\|^2_{L^2(\Omega)} = \sum_{k=1}^{+\infty}\left(\varphi_k+\frac{\varphi_k-\psi_k}{1-e^{-\lambda_kT}} \left(e^{-\lambda_k t}-1\right)\right)^2.$$
Hence, applying the inequality $(a+b)^2 \leqslant 2 a^2+2 b^2$, we obtain
$$
\|u(x, t)\|^2_{L^2(\Omega)} \leqslant 2 \sum_{k=1}^{+\infty} \varphi_k^2+2 \sum_{k=1}^{+\infty}\left(\varphi_k-\psi_k\right)^2 \left(\frac{e^{-\lambda_k t}-1}{1-e^{-\lambda_k T}}\right)^2 \text {. }
$$
Using  $\left|\frac{e^{-\lambda_k t}-1}{1-e^{-\lambda_k T}}\right| \leqslant 1$ and $(a+b)^2\leq 2a^2+b^2$, it yields that
$$
\begin{aligned}
\|u(x, t)\|^2_{L^2(\Omega)} & \leqslant 2 \sum_{k=1}^{+\infty} \varphi_k^2+2 \sum_{k=1}^{+\infty}\left(\varphi_k-\psi_k\right)^2 \\
& \leqslant 6 \sum_{k=1}^{+\infty} \varphi_k^2+4 \sum_{k=1}^{+\infty} \psi_k^2.
\end{aligned}
$$
Hence, using Bessel's inequality, we obtain $$\max _{t \in[0, T]}\|u(x, t)\|^2_{L^2(\Omega)} \leqslant 6\|\varphi(x)\|_{L^2(\Omega)}+4\|\psi(x)\|_{L^2(\Omega)}. $$

Since $\mathbb{H}_0^s(\Omega)\subset L^2(\Omega) $ from the assumptions of Theorem \ref{Theorem6.3} it follows that  $\varphi,$ $\psi \in L^2(\Omega)$. Then from the last it follows that $u \in C\left([0, T]; L^2(\Omega)\right)$.

Now, we prove that the function $u(x,t)$ defined by \eqref{6.6} belongs to the class $C\left((0,T] ; \mathbb{H}_0^s(\Omega)\right)$. Rewrite \eqref{6.6} in the form 
\begin{equation}\label{6.10}
u(x,t)=\sum_{k=1}^{+\infty} \sqrt{\lambda_k}\left(\varphi_k+\frac{\varphi_k-\psi_k}{1-e^{-\lambda_kT}} \left(e^{-\lambda_k t}-1\right)\right) \frac{\phi_k(x)}{\sqrt{\lambda_k}}.  
\end{equation}
Since $\left\{{\phi_k(x)}/{\sqrt{\lambda_k}}\right\}_{k\in \mathbb{N}}$ is an orthonormal system in $\mathbb{H}_0^s(\Omega)$ \cite{25}, we have
$$
\|u(x, t)\|^2_{ \mathbb{H}_0^s(\Omega)} = \sum_{k=1}^{+\infty} \lambda_k \left(\varphi_k+\frac{\varphi_k-\psi_k}{1-e^{-\lambda_k T}} \left(e^{-\lambda_k t}-1\right)\right)^2.
$$
Hence, applying  $(a+b)^2 \leqslant 2 a^2+2 b^2$ and  $\left|\frac{e^{-\lambda_k t}-1}{1-e^{-\lambda k T}}\right|^2 \leqslant 1$, we obtain
$$
\begin{aligned}
\|u(x, t)\|_{\mathbb{H}_0^s(\Omega)} & \leqslant 2 \sum_{k=1}^{+\infty} \lambda_k \varphi_k^2+2 \sum_{k=1}^{+\infty}\lambda_k\left(\varphi_k-\psi_k\right)^2 \\
& \leqslant 6 \sum_{k=1}^{+\infty} \lambda_k \varphi_k^2+4 \sum_{k=1}^{+\infty} \lambda_k \psi_k^2 .
\end{aligned}
$$

Since $\sqrt{\lambda}_k \varphi_k=\langle\varphi, \frac{\phi_k}{\sqrt{\lambda}_k}\rangle_{\mathbb{H}_0^s(\Omega)}$ and $ \sqrt{\lambda_k} \psi_k=\langle \psi, \frac{\phi_k}{\sqrt{\lambda_k}}\rangle _{\mathbb{H}_0^s(\Omega)}$, by applying Bessel's inequality in the latter inequality we have 
$$\max _{t \in(0, T]}\|u(x, t)\|_{\mathbb{H} _0^s(\Omega)} \leqslant 6\|\varphi\|_{\mathbb{ H}_0^s(\Omega)}+4\|\psi\|_{\mathbb{H}_0^s(\Omega)}.$$

Thus, based on the assumptions of Theorem \ref{Theorem6.3}, we conclude that $u(x,t)\in C((0,T];\mathbb{H}_0^s(\Omega))$. 

Next, we show that the series defined by \eqref{6.6} converges in $L^2(\Omega)$ uniformly for $0<t_0<T$.  By differentiating \eqref{6.6} with respect to $t$, we obtain 
\begin{equation}\label{6.11}
u_t(x, t)=\sum_{k=1}^{+\infty}  \frac{\lambda_k e^{-\lambda_kt}(\psi_k-\varphi_k)}{1-e^{-\lambda_kT}}  \phi_k(x).  
\end{equation}
Hence, with $\left\|\phi_k(x)\right\|_{L^2(\Omega)}=1$,  it implies
$$
\left\|{u_t}(x, t)\right\|^2_{L^2(\Omega)} = \sum_{k=1}^{+\infty}\left(\frac{\lambda_k t e^{-\lambda_k t}}{t}\right)^2 \left(\frac{\psi_k-\varphi_k}{1-e^{-\lambda_k T}}\right)^2.
$$
From \eqref{eq3.18} it follows that  $\displaystyle{\lambda_k t e^{-\lambda_k t} \leqslant \frac{1}{t  e}\leq \frac{1}{t_0  e}}$. Taking this into account  and using the fact $$\frac{1}{1-e^{-\lambda_k T}}<\frac{1}{1-e^{-\lambda_1 T}},$$
we obtain 
\begin{equation}\label{6.12}
\begin{aligned}
\max_{t\in (0,T]}\left\|u_t(x, t)\right\|^2_{L^2(\Omega)} & \leqslant \frac{1}{t_0^4 e^2} \frac{1}{\left(1-e^{\lambda_1T}\right)^2}  \sum_{k=1}^{+\infty} \left(\psi_k-\varphi_k\right)^2 \\
& \leqslant C_8 \left\{\sum_{k=1}^{+\infty}  \psi_k^2+  \sum_{k=1}^{+\infty} \varphi_k^2 \right\}\\
& \leqslant C_8 \left(\|\psi\|_{L^2(\Omega)}+\|\varphi\|_{L^2(\Omega)}\right),
\end{aligned}    
\end{equation}
where $C= \frac{2}{t_0^4 e^2} \frac{1}{\left(1-e^{-\lambda_1 T}\right)^2}$. 

From \eqref{6.12}, it follows that $u(x,t)\in C^1((0,T];L^2(\Omega))$. 

Similarly, one can show that the function $f(x)$ defined by \eqref{6.7} belongs to $L^2(\Omega)$ under the conditions of Theorem \ref{Theorem6.3}.

The existence of the solution of Problem \ref{problem6.2} has been proved.

Now, we show the uniqueness of the solution to Problem \ref{problem6.2}. Assume that the solution of Problem \ref{problem6.2} is not unique. That is, suppose that there exists two solutions $\left\{u_1(x, t), f_1(x)\right\}$ and $\left\{u_2(x, t), f_2(x)\right\}$ of Problem \ref{problem6.2}. Let us introduce functions
$$
u(x,t)=u_1(x, t)-u_2(x, t)
$$
and
$$
f(x)=f_1(x)-f_2(x).
$$

Then the functions $u(x, t)$ and $f(x)$ satisfy \eqref{6.4} and homogeneous conditions 
\begin{equation}\label{homogeneouscondition}
    u(x,0)=0, \quad x\in \mathbb{R}^n\backslash \Omega;  \quad u(x,T)=0, \quad  x \in \mathbb{R}^n\backslash\Omega.
\end{equation}
 We expand the functions $u(x,t)$ and $f(x)$
by the system of eigenfunctions of the spectral problem \eqref{weakformulation} as in \eqref{6.8}. 

Substituting \eqref{6.8} into \eqref{6.4}, we have 
$$\frac{ d u_{k}(t)}{dt}
+\lambda_{k} u_{k}(t)=f_{k}, \quad k\in \mathbb{N}.
$$

The general solution of this equation has the form
$$
u_k(t)=C_k e^{\lambda_{k}t}-\frac{f_k}{\lambda_{k}}, \quad k\in \mathbb{N}.
$$
where $C_k$ and $f_k$ are unknown constants. 

By the homogeneous conditions \eqref{homogeneouscondition},  we have 
$$u_k(0)=u_k(T)=0.$$

Using these we seek the unknown constants 
$C_k$ and $f_k$. First, we find $C_k$:
$$u_k(0)=C_k-\frac{f_k}{\lambda_{k}}=0, \quad \text{then} \quad C_k=\frac{f_k}{\lambda_{k}}. $$

From
$$u_k(T)=C_ke^{\lambda_{k}T}-\frac{f_k}{\lambda_{k}}=0, $$
we obtain
$$\frac{f_k}{\lambda_{k}}(e^{\lambda_{k}T}-1)=0.$$

Thus, $f_{k} \equiv 0,$ and $u_{k}(t) \equiv 0$.  Therefore , the completeness of the system $\{ \phi_{k}(x)\}_{k=1}^{+\infty}$ in $L^2(\Omega)$ follows that $f(x) \equiv 0$ and $u(x, t) \equiv 0$. It implies the uniqueness of the solution of the Problem \ref{problem6.2}.

\end{proof}

\section*{Acknowledgements}
This research was funded by the Committee of Science of the Ministry of Science and Higher Education of Kazakhstan (Grant number BR21882172). This work was also supported by the NU program 20122022CRP1601. A. Mamanazarov would like to thank Nazarbayev  University for the support during his research visit. No new data was collected or generated during the research.

\end{document}